\documentclass[12pt]{amsart}

\usepackage{amssymb, amsmath}
\usepackage{graphics}
\usepackage{epsfig}
\usepackage{amsxtra}
\usepackage{color}

\numberwithin{equation}{section}

\newtheorem{thm}{Theorem}[section]

\newtheorem{lemma}[thm]{Lemma}

\newtheorem{prop}[thm]{Proposition}

\theoremstyle{definition}

\newcommand{\defeq}{\stackrel{\rm{def}}{=}}

\renewcommand{\Re}{\operatorname{\rm Re}\nolimits}
\renewcommand{\Im}{\operatorname{\rm Im}\nolimits}

\def \comp {\operatorname{comp}}

\def \tr {\operatorname{tr}}

\def \vhalf {V^{1/2}}
\def \isig {i\sigma}

\def \restrict {\!\!\upharpoonright}

\def \Real {{\mathbb R}}

\def \Sphere {\mathbb{S}}
\def \Complex {\mathbb{C}}
\def \Natural {{\mathbb N}}
\def \Sphere {{\mathbb S}}

\def \Integers {{\mathbb Z}}

\title [Lower bounds for resonance counting functions]
{Lower bounds for resonance counting functions for 
Schr\"odinger operators with fixed sign potentials in even dimensions
}
   \author { T.J. Christiansen}
\address{Department of Mathematics,
University of Missouri,
Columbia, Missouri 65211, USA} 
\email{christiansent@missouri.edu}

\begin{document}

\begin{abstract} If $d$ is even,
the resonances of the Schr\"odinger operator $-\Delta +V$
 on $\Real^d$ with $V$ bounded and
compactly  supported  are points on $\Lambda$, the 
logarithmic cover of $\Complex \setminus \{0\}$.  We show that for
fixed sign potentials $V$ and for $m\in \Integers \setminus \{0\}$,  
 the resonance
counting function for the $m$th sheet of $\Lambda$ has maximal order
of growth.
\end{abstract}

\maketitle

\section{Introduction}

The purpose of this paper is to prove some optimal lower bounds on the 
growth rate of resonance-counting functions for certain Schr\"odinger operators
in even-dimensional Euclidean space.
The resonances associated to the Schr\"odinger operator
$-\Delta +V$, with potential $V\in L^{\infty}_{\comp}(\Real^d)$,
 lie on $\Lambda$, the 
logarithmic cover of $\Complex \setminus \{ 0 \}$, if $d$ is even.
The main result of this paper is that for  scattering by a fixed sign, 
compactly supported  
potential $V$ the resonance counting function for the 
$m$th sheet of $\Lambda$
 has maximal order of growth for any $m\in \Integers \setminus \{ 0\}$.  
Though the results of \cite{ch-hi2} show that there are many potentials
with resonance counting functions for the $m$th sheet having maximal order
of growth, the technique of \cite{ch-hi2} does not give a way of 
identifying them other than those which are scalar multiples of
the characteristic function of a ball.  In comparison, in odd dimensions
$d\geq 3$ the only specific real-valued potentials 
$V\in L^\infty_{\comp}(\Real^d)$
which are known to have resonance-counting function with optimal 
order of growth  are 
certain radial potentials \cite{zwradpot}, though in that case
asymptotics are known (see \cite{zwradpot} and \cite{dinh-vu}).
%To the best of our knowledge, the only specific potentials in even
%dimensions for which this 
%has previously been proved 
%to hold
% are some radial potentials,
 %nonzero multiples of the characteristic function of a ball,
%although from results of \cite{ch-hi2} it is known that infinitely
%many such potentials
%exist.    
%\footnote{It is likely that for a class of radial potentials such as those considered by 
%Zworski \cite{ zwradpot} for the case of odd dimension $d$ stronger results will hold.}

Let $V\in L^\infty_{\comp}(\Real^d)$ and let $\Delta\leq 0$ denote
the Laplacian on $\Real^d$.  We define the resolvent
$R_V(\lambda)=(-\Delta +V-\lambda^2)^{-1}$ for $\lambda$ in the 
``physical space'', $0<\arg \lambda <\pi$.  With at most a finite number
of exceptional values of $\lambda$, 
$R_V(\lambda)$ is bounded on $L^2(\Real^d)$ for $\lambda$ in this region. It 
is well known that for 
$\chi \in L^{\infty}_{\comp}(\Real^d)$, $\chi R_V(\lambda)\chi$ 
has a meromorphic continuation to $\Complex$ when $d$ is odd  and
to $\Lambda$, the logarithmic cover of $\Complex \setminus \{ 0\}$, 
when $d$ is even (e.g. \cite[Chapter 2]{lrb}).
In either case, the resonances are defined to be the poles of 
$\chi R_V(\lambda)\chi$ when $\chi$ is chosen to satisfy $\chi V\equiv V$.
The fact that when $d$ is even the resonances lie on $\Lambda$ makes 
them generally more difficult to study in the even-dimensional case
than in the odd-dimensional case.  

A point on $\Lambda$ can be described
by its modulus and argument, where we do not identify points which
have arguments differing by nonzero integral multiples of $2\pi$.
Thus the physical half plane corresponds to $\Lambda_0
\defeq \{ \lambda \in \Lambda:
0<\arg \lambda <\pi\}$.  Likewise, for $m\in \Integers$
 we may define the $m$th sheet to be
$$\Lambda_m \defeq \{ \lambda \in \Lambda: m\pi <\arg \lambda<(m+1)\pi\}$$ 
which is homeomorphic with the physical region and
can be identified
with the upper half plane when convenient.

Vodev \cite{vodeveven, vodev2}, following 
earlier work of Intissar \cite{intissar} studied the resonance counting function
$n_V(r,a)$, defined to be the number of resonances (counted with multiplicity,
here and everywhere)
with norm at most $r$ and argument between $-a$ and $a$.  He showed that
there is a constant $C$ which depends on $V$ but not on $r$ or $a$ so that
$$n_V(r,a) \leq C a( r^d + (\log a)^d)),\; \text{for}\; r,\; a>1.$$

 The most general
lower bound known
is due to S\'aBarreto (\cite{SaB1}, $d\geq 4$) 
and Chen (\cite{chen}, d=2): 
$$\lim \sup_{r \rightarrow \infty} \frac{\# \{ \lambda_j:\; \text{pole of $R_V(\lambda)$}\; \text{with} \;\frac{1}{r}\leq |\lambda_j|\leq r,\; |\arg \lambda_j| \leq \log r \}}{(\log r)( \log \log r)^{-p}}
=\infty \; \forall p>1$$
for any nontrivial $V\in C_c^{\infty}(\Real^d;\Real)$.
This follows the earlier work of \cite{s-t}.
  We note
that the assumption that the potential is real-valued is crucial here.  There
are explicit examples of nontrivial complex-valued potentials
$V\in L^{\infty}_{\comp}(\Real^d)$ which can be chosen
to be smooth so that the corresponding Schr\"odinger
operator $-\Delta +V$ has neither eigenvalues nor resonances \cite{autin, 
chex, iso}.

For $m\in \Integers$, let $n_m(r)=n_{m,V}(r)$ be the number of 
resonances of $-\Delta + V$ which both lie on $\Lambda_m$ and have norm 
at most $r$.  We call this the resonance counting function for 
the $m$th sheet.  It follows from Vodev's result that $n_m(r)= O(r^d)$
as $r\rightarrow \infty$.  On the other hand, lower bounds have proved 
more elusive.   The results of \cite[Theorem 1.1]{ch-hi2} show that
``generically'' for  potentials $V\in L^{\infty}_{\comp}(\Real^d)$, 
$m\in \Integers \setminus \{ 0\}$, 
\begin{equation}
\label{eq:limsup}
\lim \sup _{r\rightarrow \infty}\frac{\log n_{m,V}(r)}{\log r}=d.
\end{equation}
However, the result of \cite{ch-hi2} is nonconstructive 
in the sense that other than  potentials which are nonzero
positive scalar multiples of  the
characteristic function of a ball
and those complex-valued potentials
which are isoresonant with them \cite{iso}, 
that paper does not give a way of identifying
the particular potentials for which (\ref{eq:limsup}) holds.

The main result of this paper is the following theorem.
\begin{thm}\label{thm:lbd} Let $d$ be even.  Suppose
 $V\in L^{\infty}_{\comp}(\Real^d)$ with $V$ bounded below
by $\epsilon \chi_B$, where $\epsilon>0$ and $\chi_B$ is the characteristic
function of a nontrivial ball $B$.  Then for any nonzero $m\in \Integers$,
$$\lim \sup _{r \rightarrow \infty} \frac{\log n_{m, \pm V}(r)}{\log r} =d.$$
\end{thm}
We  note that by Vodev's result $d$ is the maximum value this limit can obtain.
When this limit is $d$, we say that the $m$th counting function has maximal 
order of growth.

Theorem \ref{thm:lbd}, 
when combined with  \cite[Theorem 3.8]{ch-hi2}, 
has the following theorem 
as an 
immediate corollary.
\begin{thm}
Let $d$ be even, and $K\subset \Real^d$ be a compact set with nonempty interior.
Let $F$ denote either $\Real $ or $\Complex$.  Then for $m\in \Integers$, $m\not
= 0$,
the set
$$\{ V\in C^{\infty}(K;F): \lim\sup_{r\rightarrow \infty } \frac{ \log n_{m,V}(r)}{\log r}=d\}$$
is dense in $C^{\infty}(K;F)$.
\end{thm}
For the case of {\em odd} dimension $d\geq 3$, the analog of this theorem
was proved in \cite{scv}.  A stronger result holds in dimension 
$d=1$, see \cite{zworski1d} or \cite{froese,regge,simon}.

Theorem \ref{thm:lbd} may be compared with other results for 
fixed-sign potentials.
In the odd-dimensional case Lax-Phillips \cite{l-p} and Vasy \cite{vasy}
proved lower bounds on the number of pure imaginary resonances for potentials
of fixed sign.  In \cite{ch-hi4} it is shown that in even dimensions there
are no ``pure imaginary'' resonances for positive potentials, and on each
sheet $\Lambda_m$ of $\Lambda$, only finitely many for negative 
potentials.  Both
\cite{l-p} and \cite{vasy} use a monotonicity property for potentials of 
fixed sign.  This paper also uses a monotonicity property, though it is 
more closely related to one used in \cite{scv}.  Also important here are 
some results from one-dimensional 
complex analysis, more delicate than the corresponding complex-analytic
arguments from
\cite{scv}.

{\bf Acknowledgments.}  The author 
gratefully acknowledges the partial support of the NSF under grant 
DMS 1001156.

\section{Some Complex Analysis}

The main result of this section is
 Proposition \ref{p:complex}, which, roughly speaking, 
controls the growth of a function $f$ 
analytic in a half plane in terms of the growth of 
the counting function for the 
 zeros of $f$ in the half plane and the behavior of 
$f$ on 
the boundary of the half plane.

Both the statement and the proof of the following lemma
 bear some resemblance to those for
 Carath\'eodory's inequality for the disk. 
 The estimate we obtain here is likely a 
crude one, but suffices for our purposes.
\begin{lemma}\label{l:caratheodorytype}
Let $f$ be analytic in a neighborhood of 
$$\Omega_R\defeq \{ z\in \Complex: 
1\leq |z| \leq 
R,\; \Im z\geq 0\},$$
 $\rho>0$,  and for $x\in \Real \cap \Omega_R$, $|f(x)|\leq C_0|x|^{\rho}$
for some constant $C_0>0$.  Set 
$$M=\max_{|z|=1, \; z\in \Omega_R}|f(z)|$$
and define
$$A(R)= \max\left(   C_0 R^\rho,
 M R^\rho , \max_{ z\in \Omega_R }\Re f(z) \right).$$
Then
if $1<r<R$  and $z\in \Omega_R$ 
with $|z|=r$, then $|f(z)|\leq \frac{2r^{\rho}}{R^\rho- r^\rho} A(R).$
\end{lemma}
\begin{proof}
Set 
$$g(z)= \frac{1}{z^\rho} \frac{f(z)}{2A(R)-f(z)}$$
which is analytic in a neighborhood of $\Omega_R$.  We bound $|g|$ on the 
boundary of $\Omega_R$.   If  $z\in \Omega_R$ has $|z|=R$,
then $$|g(z)|\leq \frac{1}{R^\rho} \frac{|f(z)|}{|2A(R)-f(z)|}\leq \frac{
|f(z)|}{R^\rho |f(z)|}=\frac{1}{R^\rho}.$$  
Notice that if 
$x\in \Omega_R\cap \Real$, 
$$|g(x)|\leq \frac{1}{|x|^\rho}\frac{C_0 |x|^\rho}{R^\rho C_0} \leq 
\frac{1}{R^\rho}.$$
Moreover, if $z\in \Omega_R$ has $|z|=1$, since 
$A(R)\geq |f(z)|R^\rho$, $|g(z)|\leq 1/R^\rho$.  
   Thus, by the maximum principle 
$|g(z)|\leq 1/R^\rho$ for all $z\in \Omega_R$.  

Suppose $z\in \Omega_R$ with $|z|=r$, $1<r<R$.  Then
$$|f(z)|\leq |z|^\rho |2 A(R)-f(z)| R^{-\rho}\leq  r^\rho (2A(R)+|f(z)|) R^{-\rho}.$$
Rearranging, we find 
$$(R^\rho -r^\rho )|f(z)|\leq 2 r^\rho A(R),$$
or $$|f(z)|\leq \frac{ 2 r^\rho}{R^\rho -r^\rho}A(R).$$
\end{proof}

\begin{lemma} \label{l:p-lapp}
Let $\Omega=\{ z:\; \Im z\geq 0,\;|z|\geq 1\}$ and suppose $f$ is analytic
in a neighborhood of $\Omega$, and there are  constants $\rho_0$, $C_0$, so that
$|f(z)|\leq C_0 \exp (C_0 |z|^{\rho_0})$ for all $z\in \Omega$.  Suppose there are 
constants $C_1$, $\rho>0$ so that for all 
$x>1$,  
%$|f(x)|\leq C_1\exp( |x|^\rho)$ and 
$|\int_1^x f'(t)/f(t)dt|\leq C_1 |x|^\rho$ and for all
$x<-1$,  $|\int_x^{-1} f'(t)/f(t)dt|\leq C_1 |x|^\rho$.
  If, in addition, $f$ does not vanish 
in $\Omega$, then  there is a constant $C_3$ so that $|f(z)|\leq 
C_3 \exp(C_3|z|^\rho)$ for all
$z\in \Omega$.
\end{lemma}
\begin{proof}In the proof we 
shall denote by $C$ a constant the value of which may change from line to line without comment.  

Since $f$ is nonvanishing in $\Omega$, there is a function $g$ analytic on
$\Omega$ so that $\exp g(z)= f(z)$. 
Since $g'(z)= f'(z)/f(z)$,  
$$g(x)-g(1)= \int_1^x \frac
{f'(t)}{f(t)} dt\; \text{if}\; x>1$$
so that $|g(x)|\leq C |x|^\rho+|g(1)|$ when $x\geq 1$ for
some constant $C$.  A similar argument gives a similar
 bound for $x\leq -1$.  

We now assume $\rho < \rho_0$ since otherwise there is nothing to prove.
We give a  bound on the growth of $g$ at infinity which is more 
than adequate to allow us to apply a version of the Phragm\'en-Lindell\"of theorem 
as we will below.
Since $\Re g(z)=\log |f(z)|$, for all $z\in  \Omega$,
$\Re g(z) \leq  C (1+ |z|^{\rho_0})$.
% for some constant $C$. 
Applying 
Lemma \ref{l:caratheodorytype}, we find that 
$|g(z)|\leq C (1+|z|^{\rho_0})$ for all $z\in \Omega$.  

Consider the function $h(z)= g(z) / (i+z)^{\rho}$.  This is an analytic 
function in a neighborhood of $\Omega$ and is bounded on $\partial \Omega$.
Thus, by a version of the Phragm\'en-Lindell\"of theorem
(proved, for example, by an easy modification of the proof of 
\cite[Corollary VI.4.2]{conway} using \cite[Theorem VI.4.1]{conway}), $h$
is bounded in $\Omega$.  This implies that for 
$z\in \Omega$, $|f(z)| = \exp (\Re g(z)) 
\leq \exp (C(1+|z|)^{\rho})$ for some constant $C$.
\end{proof}

%{\em note:}
%I need to check here if the condition that $\rho$ is not an integer is necessary.  In applications, $\rho$ will not be an integer, so this does not 
%affect the usefulness of this proposition for me, but it's nice to be
%precise.

We shall use the notation 
$E_0(z) =1-z$ and $E_p(z) = (1-z)\exp(z+z^2/2+...+ z^p/p)$ for $p\in \Natural$
for a canonical factor. The proof of the following lemma bears many similarities
to proofs for estimates of canonical products; see for 
example \cite[Lemma I.4.3]{levin}. 

\begin{lemma}\label{l:cproductreals}
Let $\{a_j\}\subset \Complex$ be a set of not necessarily 
distinct points in the open 
upper half plane, with $|a_1|\leq |a_2|\leq ...$ and suppose for 
some constants $C_0$ and $\rho$
$$n(r)\defeq \# \{ j :\; |a_j|\leq r \}
 \leq C_0 r^\rho\; \text{ when} \; r\geq 1.$$
Suppose $\rho > 0$ is not an integer, let $p$ be the greatest
integer less than 
$\rho$, and 
set 
$$f(z)=\prod_{n=1}^{\infty} \frac{E_p(z/\overline{a_n})}{ E_p(z/a_n)}.$$
Then for $x\in \Real$
$$\left|\int_0^x \frac{f'(t)}{f(t)} dt\right| = O(|x|^\rho)$$
as $|x|\rightarrow \infty$.
\end{lemma} 
\begin{proof}
We note first that
our assumption on $n(r)$ ensures that the canonical 
products converge, so that $f$ is a meromorphic function on $\Complex$.
Moreover, by assumption $f$ has neither poles nor zeros on the real line.

A computation shows that $E_p'(z)/E_p(z)= -z^p/(1-z)$.  Thus
\begin{equation}\label{eq:logderiv}
\frac{f'(x)}{f(x)}= \sum_{n=1}^{\infty}\left( \frac{(x/\overline{a_n})^p}{x-\overline{a_n}}-\frac{(x/a_n)^p}{x-a_n}\right).
\end{equation}

Let $a\in \Complex$, $\Im a >0 $, $t\in \Real$.
 Then
\begin{align}\label{eq:basic}
 \frac{(t/\overline{a})^p}{t-\overline{a}}-
\frac{(t/a)^p}{t-a} & = \frac{1}{|a|^{2p}}
 t^p \left(\frac{ a^p(t-a)-\overline{a}^p (t-\overline{a})}{|t-a|^2}\right)
\nonumber
\\ & = 2i \frac{1}{|a|^{2p}}
 t^p \left( \frac{ t \Im (a^p) - \Im (a^{p+1})}{|t-a|^2}\right).
\end{align}
Now set $a=\alpha + i\beta$, $\beta>0$, and
note that 
$|\Im a^p| \leq p\beta |a|^{p-1}$.  Thus  for $x\in \Real$ 
\begin{align}\label{eq:intbd}
\left| \int_0^x \left( \frac{(t/\overline{a})^p}{t-\overline{a}}-
\frac{(t/a)^p}{t-a}\right) dt\right| & 
\leq \frac{2}{|a|^{p+1}}  \left|
\int_0^x \frac{ p|t|^{p+1}\beta + (p+1) |a||t|^p\beta}{(t-\alpha)^2+ \beta^2} dt\right|.
\end{align}
Now for $q>0$
\begin{equation*}
 \int_0^x  \frac{ t^{q}\beta }{(t-\alpha)^2+ \beta^2} dt
= x^{q} \arctan((x-\alpha)/\beta)- q\int_0^x
t^{q-1} \arctan((t-\alpha)/\beta) dt.
\end{equation*}
Using that for $s\in \Real$, $|\arctan s| <\pi/2$, we find that for $|x|>1$
\begin{equation}\label{eq:est1}
\left|\int_0^x  \frac{ t^{q}\beta }{(t-\alpha)^2+ \beta^2} dt \right|
\leq C |x|^{q}
\end{equation}
for some constant $C$, independent of $\alpha$ and $\beta$.  
%Likewise
%\begin{equation}\label{eq:est2}
%\left|\int_0^x  \frac{ t^{p}\beta |a|}{(t-\alpha)^2+ \beta^2} dt\right| 
%\leq C|a| |x|^p.
%\end{equation}

To prove the lemma, we will split  $\{a_j\}$ into two sets, depending on 
the relative size of  $|a_j|$ and $2|x|$.  
For $|a_j|\leq 2|x|$, we 
first note that 
$$p\int _0^x \frac{ \beta t^{p+1}}{(t-\alpha)^2+\beta^2}dt 
= p\beta \int _0^x t^{p-1}\left( 1+ \frac{ 2\alpha t -(\alpha^2+\beta^2) }
{(t-\alpha)^2+\beta^2}\right)dt$$
and
use (\ref{eq:intbd}) and (\ref{eq:est1})  to
get
\begin{align}\label{eq:smalla}
\sum_{|a_j|\leq 2|x| } \left| 
\int_0^x 
\left( \frac{(t/\overline{a_j})^p}{t-\overline{a_j}}-
\frac{(t/a_j)^p}{t-a_j}\right) dt\right|
&  \leq C  \sum_{|a_j|\leq 2|x|}
\left( |x|^p |a_j|^{-p} +  |x|^{p-1}|a_j|^{-p+1}
\right).
\end{align}
Since 
\begin{equation}\label{eq:apowersum}
\sum_{1\leq |a_j|\leq r}|a_j|^{q} = \int_1^r t^{q}dn(t)= r^qn(r)-n(1)-\int_1^r
qt^{q-1} n(t)dt
\end{equation}
applying our upper bound on $n(r)$ we get from (\ref{eq:smalla})
\begin{equation}\label{eq:asmallbound}
\sum_{|a_j|\leq 2|x| } \left| 
\int_0^x\left( \frac{(t/\overline{a_j})^p}{t-\overline{a_j}}-
\frac{(t/a_j)^p}{t-a_j}\right) dt\right|\leq
C(|x|^{\rho}+1).
\end{equation}

Now we bound the contribution of the $a_j$ with $|a_j|> 2 |x|$.  
For this we use (\ref{eq:basic}) more directly.  Here
\begin{align*}  & 
\sum_{|a_j|>2|x|}
\left| \int_0^x 
\left( \frac{(t/\overline{a_j})^p}{t-\overline{a_j}}-\frac{(t/a_j)^p}{t-a_j}\right) dt \right|   \\ & 
= 2 \sum_{|a_j|>2|x|}
\left| \int_0^x 
 \frac{1}{|a_j|^{2p}}
 t^p \left( \frac{ t \Im (a_j^p) - \Im (a_j^{p+1})}{|t-a_j|^2}\right)dt\right|\\
& \leq C  \sum_{|a_j|>2|x|}
\left| \int_0^x \frac{1}{|a_j|^{2p}}
  \left( \frac{ |t|^{p+1} |a_j|^p + |t|^p|a_j|^{p+1})}{|a_j|^2}\right) dt \right|
\\ & \leq C  \sum_{|a_j|>2|x|} (|x|^{p+2} |a_j|^{-p-2}+|x|^{p+1}|a_j|^{-p-1}).
\end{align*}
Applying the analog of (\ref{eq:apowersum}) and using the upper bound on
$n(r)$  we obtain
$$\sum_{|a_j|>2|x|} |a_j|^{-q} \leq C|x|^{\rho-q}$$
provided $q>\rho$, giving us
$$\sum_{|a_j|>2|x|}
\left| \int_0^x 
\left( \frac{(t/\overline{a_j})^p}{t-\overline{a_j}}-\frac{(t/a_j)^p}{t-a_j}\right) dt \right| \leq C|x|^\rho.$$
Combined with (\ref{eq:asmallbound}), this completes the proof of the lemma.
\end{proof}

%The following lemma is very similar to \cite[Theorem I.12]{levin}, which
%is a result for entire functions.  We include a proof, similar to the 
%proof of that theorem, for the reader's convenience.
%\begin{lemma}  Let $g$ be an entire function of order $\rho$, and
%$h$ be a function analytic in $\Omega = \{ z\in \Complex: \Im z \geq 0, \; 
%|z|\geq 1\}$.  If $gh$ is of order $\rho'$ in $\Omega$, then the order
%of $h$ in $\Omega$ is $\max(\rho,\rho')$.
%\end{lemma}
%\begin{proof}
%It is immediate that $\rho'$ is less than or equal to the 
%maximum of the orders of $h$ and $g$.  

\begin{prop}\label{p:complex}
Let $f$ be a function analytic in a neighborhood of $
\Omega=\{ z:\: \Im z \geq 0, \; |z|\geq 1 \}$. Suppose $f$ does not vanish on $\Real \cap \Omega$,
and
let $n(r)$ be the number of zeros of $f$ in 
$\{ z:\; \Im z \geq 0, \; 1\leq |z|\leq r\}$
counted with multiplicity.
 Suppose that that there are constants $C_0$ and 
$\rho$, $\rho$ not an integer, so that 
$$n(r) \leq C_0(1+r^\rho )$$
and
$$\left| \frac{f'(x)}{f(x)}\right| \leq C_0 (1+|x|^{\rho-1}) \; \text{ for all } 
x\in \Real \; \text{with}\; |x|\geq 1.$$  Suppose in addition that there are some constants $\rho_1$, $C_1$
so that $\log |f(z)|\leq C_1(1+|z|^{\rho_1}) $ for all $z\in \Omega$.
Then there is a constant $C$ so
that $|f(z)|\leq Ce^{|z|^\rho} $ for $z\in \Omega$.
%is of order at most $\rho$ in 
%$\{ z:\: \Im z \geq 0,\; \text{and}\; z\not = 0\}$. 
\end{prop}
\begin{proof}

We will assume $\rho_1>\rho$ as otherwise there is nothing to prove.

To aid in notation, we set 
$$\Omega_R=
 \{ z\in \Complex: \; \Im z \geq 1\; \text{and}\; 1\leq |z|\leq R\}.$$

We prove this proposition by constructing a function to which we can 
apply Lemma \ref{l:p-lapp}.
Let $p$ denote the greatest integer less than 
%or equal to 
$\rho$, and 
$\{ a_j\}$ the zeros of $f$ in $\Omega$, repeated according to 
multiplicity, with $|a_1|\leq |a_2|\leq ...$.  Set
$$h(z)=\frac{f(z)g_1(z)}{g_2(z)},$$ 
where 
$$g_1(z)= \prod_{n=1}^{\infty}  E_p(z/\overline{a_n}) \; \text{and}\; g_2(z)=
\prod_{n=1}^{\infty} E_p(z/a_n).$$
%where 
%$E_0(z) =1-z$ and $E_p(z) = (1-z)\exp(z+z^2/2+...+ z^p/p)$ if $p>0$.
Note that $h$ is analytic in $\Omega$ and does
not vanish there.  

 As an intermediate step we show that $\log |h(z)|\leq C|z|^{\rho_1} $
for all $z\in \Omega$.  Recall we have
assumed $\rho_1>\rho$.  Here and below $C$ is a finite constant which
may change from line to line.
If $x\in \Real$, $1\leq |x|$, then $\log|h(x)|=\log|f(x)|\leq C_0(1+|x|^{\rho_1})$.
Moreover, from estimates on canonical products,
\begin{equation}\label{eq:gjbound}
\log |g_j(z)| \leq C(1+|z|^\rho), \; j=1,2
\end{equation}
for some constant $C$, see \cite[Lemma I.4.3]{levin}.   

To aid in notation, we set $\Omega_R= \{ z\in \Complex: \; \Im z \geq 1\; \text{and}\; 1\leq |z|\leq R\}$.

As is shown in the proof of \cite[Theorem I.12]{levin}, given $R>0$ and 
$0<\delta<1$ there is an $r_j\in [R, R(1-\delta)^{-1}]$ so that
for all $z\in \Complex$ with $|z|=r_j$, 
\begin{align} 
\log |g_j(z)| & \geq -\left(2+\log \frac{12 e}{\delta}\right) \log
\max_{|z|=2 e R(1-\delta)^{-1}} |g_j(z)|,\; \text{for $j=1,\;j$}.
\end{align}
Using (\ref{eq:gjbound}), this gives
\begin{align}\label{eq:g2lb}
\log |g_j(z)|
& \geq -C_{ \delta} (1+(R(1-\delta)^{-1})^\rho ),\; |z|=r_j,\; j=1,\;2.
\end{align}
Now fix a $\delta>0$, $\delta<1$.  Given any $R>1$, we can find an 
$r_2\in [R,R(1-\delta)^{-1}]$ as 
above so that (\ref{eq:g2lb}) holds for $j=2$.  
Then using that
$$\max_{z\in \Omega_R } \log |h(z)|
\leq \max_{ z\in \Omega_{r_2}} \log |h(z)|= \max_{z\in \partial \Omega_{r_2}}
\log |h(z)|,
$$
 our assumptions on $f$, and (\ref{eq:g2lb}) we find
$$\max_{ z\in \Omega_R}\log |h(z)|
\leq C_\delta (1+ (R(1-\delta)^{-1})^{\rho_1})+ C(1+R^\rho)\leq \tilde{C}_\delta (1+R^{\rho_1}).$$
%That is, the order of $h$ in the upper half plane does not exceed $\rho_1$. 

%Thus $g_1$, $g_2$ are entire functions of order at most $\rho$, 
%and $hg_2=fg_1$ has order (in the upper half plane) at most the maximum of the %order of $f$ and $g_1$, that is, at most $\rho_1$ in the upper half plane.  
%Since $h$ is analytic in the upper half plane,
%it has order at most $\rho_1$ there.  {\em Warning:} This is not quite right,
%and the argument that the order of $h$ in the UHP is finite still must
%be fixed.  (We only need show that it is finite, we really haven't quite
%fixed this yet; see below as well.)

 %(see, for
%example, \cite[Theorem 1.7]{levin}, the order of $h$ in 
%$\{ z: \Im z\geq 0,\; |z|>0\}$ does not exceed $\rho$.  
                        
For $x\in \Real$, $|x|\geq 1$, 
\begin{align*}
\frac{h'(x)}{h(x)}& =\frac{f'(x)}{f(x)}+ 
\frac{(g_1/g_2)'(x)}{g_1(x)/g_2(x)}.
\end{align*}
By applying our assumptions on $f$ and Lemma \ref{l:cproductreals}, we find
that for $x>1$, $|\int_1^x h'(t)/h(t) dt| = O(x^\rho)$, and likewise
for $x<-1$, $|\int _{x}^{-1} h'(t)/h(t) dt|=O(|x|^{\rho})$.  By Lemma 
\ref{l:p-lapp}, there is a constant $C$ so that
\begin{equation}\label{eq:hbd}
\log |h(z)|\leq C(1+ |z|^\rho),\;  \text{when}\; z\in \Omega.
\end{equation}

Now we write 
$f(z) = g_2(z)h(z)/g_1(z)$, holomorphic in a neighborhood of
$\Omega$.
Given $R>1$ and $\delta$ satisfying $0<\delta<1$, as above
we choose $r_1\in [R,R(1-\delta)^{-1}]$
so that (\ref{eq:g2lb}) holds for $g_1$.  Using in addition
 (\ref{eq:gjbound}) and
(\ref{eq:hbd}), we find there is a constant so that
$$\log |f(z)|\leq C(1+(R(1-\delta )^{-1})^{\rho})\; \text{for}\; |z|=r_1,
\; \Im z\geq 0. $$  As in the proof of the bound on $h$, since 
$|h(x)|=|f(x)|$ for $x\in \Real \cap \Omega_R$, we find then that there
is a constant $C$ so that
$$\max_{z\in \Omega_R}\log |f(z)|\leq C(1+R^{\rho}).$$ 

%Recall $f(z)= h(z)/g(z)$,  $g$ is nonvanishing in the upper
%half plane, and $|g(x)|=1$ for $x\in \Real$, $x\not =0$.  

%Recall that the zeros of $g_1$ all lie in the open lower half plane.  Thus there is a function
%$\varphi$, analytic in the upper half plane, so that $\exp(\varphi(z))=g_1(z)$ there, and 
%$\Re \varphi(z)=\log|g_1(z)|$.  Thus by (\ref{eq:gjbound})
%$\Re \varphi(z)\leq C(1+|z|^{\rho})$ in the closed upper half plane.
%Applying Lemma \ref{l:caratheodorytype}, we see that $|\varphi(z)|\leq C(1+|z|^{\rho})$ in the upper 
%half plane, and hence $1/|g_1(z)| \leq  C\exp (C|z|^\rho)$ there.  
%Thus $f(z)=h(z)g_2(z)/g_1(z)$ is of order at most $\rho$ in the upper half plane.
\end{proof}

%***

%The function $g$ is the ratio of two canonical products, each of 
%order at most $\rho$.  All the zeros of $g$ lie
%in the lower half plane.  Moreover, $|g(x)|=1$ for 
%$x\in \Real$.  Using in addition \cite[Theorem I.11]{levin}, which gives some
%bound from below on the 
%minimum modulus of an analytic function, there 
%is a sequence of values of $r$ tending to infinity and a constant
%$C$ so that 
%$1/|g(z)|\leq C(1+r^\rho)$ for all $z$ in the upper half plane with $|z|=r$.
%({\em worry}-- Perhaps I only know that the canonoical products are of order 
%$\rho$-- I've actually assumed slightly more here.  Better check...)  Let 
%$F\subset \Real_+$ be the set of values of $r$ for which this 
%lower bound holds.  From \cite[Theorem I.11]{levin}, by perhaps making
%the constant $C$ larger, $\text{measure}(F\cap[0,R])/R>7/8$
%for all values of $R$ sufficiently large.  Thus applying the maximum 
%principle in the upper half plane, we can (with perhaps a larger value of $C$)
%ensure $1/|g(z)|\leq C(1+|z|^\rho)$ in the upper half plane and so
%$1/g$ is of order at most $\rho$ in the upper half plane.  Thus, since
%$f(z)=h(z)/g(z)$, we obtain the bound on the order of $f$ as desired.

%{\em need to doublecheck some things here!}

\section{A scalar function having zeros at the poles of the resolvent}
\label{s:defF}

We recall the derivation of some identities commonly used in 
the study of resonances for Schr\"odinger operators.
Let $V\in L^{\infty}_{\comp}(\Real^d)$ and let $d\geq 2$ be even. 
 There is no need to make an 
assumption on the sign of $V$ here.
We recall the notation $R_V(\lambda)=(-\Delta +V-\lambda^2)^{-1}$ when
$\lambda \in \Lambda_0$.  For such $\lambda$, 
$(-\Delta +V-\lambda^2)R_0(\lambda)= I+V R_0(\lambda)$
and by meromorphic continuation, $$R_0(\lambda) = R_V(\lambda)(I+VR_0(\lambda)), \; \lambda \in \Lambda.$$  Thus $R_V(\lambda)$ has a pole if and only if
$I+VR_0(\lambda) $ has a zero, and multiplicities agree.  Writing 
$V^{1/2}=V/|V|^{1/2}$ with the convention
that $V^{1/2}=0$ outside the support of $V$,
 we see that $I+VR_0(\lambda)$ has a zero if and only
if $I+V^{1/2}R_0(\lambda)|V|^{1/2}$ has a zero.  Consequently,
 $I+V^{1/2}R_0(\lambda)|V|^{1/2}$ is invertible for all but a finite number 
of points in $\overline{\Lambda_0}$.  
Thus, if $m\in \Integers$, $\lambda\in \overline{\lambda_0}$,
\begin{multline}I+V^{1/2}R_0(e^{im\pi}\lambda)|V|^{1/2}\\
= (I+V^{1/2}R_0(\lambda)|V|^{1/2})\left(I+(I+V^{1/2}R_0(\lambda)|V|^{1/2})^{-1}
V^{1/2}\left(R_0(e^{im\pi}\lambda)-R_0(\lambda)\right)|V|^{1/2}\right).
\end{multline}
But when $d$ is even
$$R_0(e^{im\pi}\lambda)-R_0(\lambda)= imT(\lambda)$$
with 
\begin{equation}
(T(\lambda)f)(x)=\alpha_d \lambda^{d-2}\int_{\Real^d} \int_{\Sphere^{d-1}}
e^{i\lambda (x-y)\cdot \omega} f(y) d\omega \; dy 
\end{equation} for $f\in L^{2}_{\comp}(\Real^d)$, with 
$\alpha_d =(2\pi)^{1-d}/2$, \cite[(1.32)]{lrb}. 
Moreover, $V^{1/2}T(\lambda) |V|^{1/2}$ is trace class. 
Thus, with at most a finite number of 
exceptions, the poles of $R_V(e^{im\pi}\lambda)$ with $\lambda \in \Lambda_0$
correspond, with multiplicity, to the zeros of 
\begin{equation}\label{eq:FmV}
F_{m,V}(\lambda) \defeq  \det \left(I+im(I+V^{1/2}R_0(\lambda)|V|^{1/2})^{-1}
V^{1/2}T(\lambda) |V|^{1/2}\right)
\end{equation}
in $\Lambda_0$.

\section{Lower bounds on $F_{m,\pm V}(i\sigma)$ when 
$V$ has fixed sign}

 In the remainder of this paper we assume $d\geq 2$ is even.

Let $V\geq 0$, $V\in L^{\infty}_{\comp}(\Real^d)$. In this
section we study the function $F_{m,\pm V}$ from (\ref{eq:FmV}).
For $\sigma \in \Real_+$, we shall use the shorthand $i\sigma$ to denote the 
point in the physical region with norm $\sigma$ and argument $\pi/2$. 
  Taking the positive
sign, $I+V^{1/2}R_0(i\sigma)|V|^{1/2}=I+V^{1/2}R_0(i\sigma)V^{1/2} $ is a positive
operator for $\sigma>0$.   When we choose the negative sign, 
we will additionally assume that $\sigma$ is 
chosen large enough that $I- \vhalf R_0(i\sigma) \vhalf $ is a positive
invertible operator; this is 
possible by, for example, insisting $\sigma>2(\| V \|_{\infty}+1)$ since
$\|R_0(i\sigma)\| \leq 1/\sigma^2$.  
With these assumptions on $\sigma$, using the properties of
 the determinant
and the fact that $V\geq 0$ we may rewrite the 
function $F_{m,\pm V}(i\sigma)$ from (\ref{eq:FmV}) as
\begin{align}\label{eq:Fmrewrite}
F_{m}(i\sigma)& =F_{ m,\pm  V}(i\sigma)\\ &= 
\det \left( I \pm im (I\pm \vhalf R_0(i\sigma) \vhalf )^{-1/2} \vhalf T(i\sigma)
 \vhalf 
(I\pm \vhalf R_0(i\sigma) \vhalf )^{-1/2} \right). \nonumber
\end{align}
%When we choose the negative sign, we will additionally assume that $\sigma$ is 
%chosen large enough that $I- \vhalf R_0(i\sigma) \vhalf $ is invertible; this is 
%possible by, for example, insisting $\sigma>2\| V \|_{\infty}$. 
 We shall 
obtain a lower bound on $F_m(i\sigma)$ as $\sigma \rightarrow \infty$.

The following proposition is central to the proof of Theorem \ref{thm:lbd}
 and is the main 
result of this section.  Related results were obtained in
odd dimensions in \cite[Section 5]{scv}. 
\begin{prop}\label{p:lb} Let $V\in L^{\infty}_{\comp}(\Real^d)$, $V\geq 0$, and let $V$ be bounded below
by $\epsilon \chi_{B}$ where $\epsilon>0$ and $\chi_B$ is the characteristic function of 
a nontrivial open ball. Let $m\in \Integers$, $m\not =0$.
 Then there is a constant $c_0>0$ so that $|F_{m,\pm V}(i\sigma)|\geq 
c_0\exp(c_0 \sigma ^d)$ for all sufficiently large $\sigma>0$.
\end{prop}  The proof is similar to the proofs of some results of
\cite{l-p, vasy} in that it uses both a property
of  monotonicity in $V$  and the fact that for potentials which are 
positive multiples of the characteristic function of a ball much
can be said by using a decomposition into spherical harmonics and special
functions.  However, the implementation of these underlying
 ideas is rather different here.

The proof of Proposition \ref{p:lb} uses the following lemma,
 a monotonicity result reminiscent of 
results of 
\cite{l-p,vasy}.  In fact, the proof of this lemma uses a result from
\cite{vasy}.
\begin{lemma}\label{l:normbd}  Let $V_1,\; V_2\in L^{\infty}(\Real^d)$ and 
suppose the support of $V_j$ is contained in $\overline{B}(R)=
\{ x\in \Real^d: |x| \leq R\}$ for $j=1,\;2$. 
Suppose $V_2(x)\geq V_1(x) \geq 0$ for all $x\in \Real^d$.  We use the 
convention that 
$\vhalf_1/\vhalf_2$ is $0$ outside the support of $V_1$.  Then 
$$\left\| (I + \vhalf_1 R_0(i\sigma) \vhalf_1)^{-1/2}\frac{ \vhalf_1}{\vhalf_2}
(I + \vhalf_2 R_0(i\sigma) \vhalf_2)^{1/2} \right\| \leq 1.$$  
If $\sigma \geq 2(\|V_2\|_{\infty}+1)$, then 
$$\left\| (I - \vhalf_1 R_0(i\sigma) \vhalf_1)^{-1/2} \frac{\vhalf_1}{\vhalf_2}
(I - \vhalf_2 R_0(i\sigma) \vhalf_2)^{1/2} \right\| \leq 1.$$ 
\end{lemma}
\begin{proof}
When $\sigma>0$ is sufficiently large that
$I\pm \vhalf_j R_0(i\sigma) \vhalf_j$ is a positive operator,
\begin{align*}
& (I\pm V_j R_0(i\sigma) )\vhalf_j (I\pm \vhalf_j R_0(i\sigma)\vhalf_j)^{-1}\\
& 
 = \vhalf_j (I\pm \vhalf_j R_0(i\sigma)\vhalf_j)
(I\pm \vhalf_j R_0(i\sigma)\vhalf_j)^{-1}
\\& = \vhalf_j.
\end{align*}
Thus
 $$\vhalf_j(I\pm \vhalf_j R_0(i\sigma) \vhalf_j)^{-1} 
\vhalf_j = (I \pm V_j R_0(i\sigma))^{-1}V_j,\; j=1,2 $$
for $\sigma>0 $ sufficiently large.
Applying \cite[Lemma 2.2]{vasy}, and using that $V_2\geq V_1$, we get
$$(I + V_2 R_0(i\sigma))^{-1}V_2\geq (I + V_1 R_0(i\sigma))^{-1}V_1.$$
When we take the ``$-$'' sign, again applying \cite[Lemma 2.2]{vasy},
$$(I - V_2 R_0(i\sigma))^{-1}V_2\geq (I - V_1 R_0(i\sigma))^{-1}V_1
$$
when $\sigma>2( \|V_2\|_\infty +1)$.   Here we note our convention differs
somewhat from \cite{vasy}, in that we take $V_j\geq 0$.   
Summarizing, 
$$\vhalf_2(I \pm \vhalf_2 R_0(i\sigma) \vhalf_2)^{-1} \vhalf_2
\geq \vhalf_1(I \pm \vhalf_1 R_0(i\sigma) \vhalf_1)^{-1} \vhalf_1$$
when $\sigma>0$ (for the ``$+$'' sign) or $\sigma>2(\|V\|+1)$ (for the 
``$-$'' sign).  For the remainder of the proof, we shall assume $\sigma>0$ 
satisfies these requirements and suppress the argument $i\sigma$.

Now let $\chi_{V_2}$ be the characteristic function of the support of 
$V_2$ and recall $\vhalf_1 \chi_{V_2}=\vhalf_1$ and note that 
$\chi_{V_2}(I\pm \vhalf_2 R_0\vhalf_2)= (I\pm \vhalf_2 R_0\vhalf_2)\chi_{V_2}$. 
 Then
$$\chi_{V_2} (I \pm \vhalf_2 R_0 \vhalf_2)^{-1} \chi_{V_2} 
\geq  \frac{ \vhalf_1}{\vhalf_2}(I \pm \vhalf_1 R_0 \vhalf_1)^{-1} \frac{\vhalf_1}
{\vhalf_2}.$$
This implies
$$\chi_{V_2} \geq (I \pm \vhalf_2 R_0 \vhalf_2)^{1/2}
\frac{ \vhalf_1}{\vhalf_2}(I \pm \vhalf_1 R_0 \vhalf_1)^{-1} \frac{\vhalf_1}
{\vhalf_2}(I \pm \vhalf_2 R_0 \vhalf_2)^{1/2}.$$
This proves the lemma, since the norm of the right hand side is the square of
the norm of the operator in question.
\end{proof}

\begin{lemma}\label{l:eigenvaluebd} Let ${\mathcal H}$ be an 
infinite dimensional complex separable Hilbert space,
 $A,\; B\in {\mathcal L }({\mathcal H})$, with $B=B^*$, 
and $\|A\| \leq 1$.  Let $|\lambda_1|\geq |\lambda_2|\geq...$ be the 
norms of the eigenvalues
of $A^*BA$, and $|\mu_1|\geq |\mu_2|\geq...$ be the norms of the 
eigenvalues of $B$. In both
cases we repeat according to multiplicity.
Then $|\mu_j|\geq |\lambda_j|$ for all $j$.
\end{lemma}
\begin{proof}
One way to prove this it that by noting that since $B$ and $A^*BA$ are
self-adjoint, the norms of the 
the eigenvalues are  the characteristic values.  Then this 
lemma is an immediate application of the bound for the characteristic
values of a product found, for example, in \cite[Theorem 1.6]{simonti}.
\end{proof}

The next lemma shows that $F_{m,\pm V}(i\sigma)$ depends monotonically
on $V$ in some sense.
\begin{lemma}\label{l:mono} Let $V_1,\; V_2\in L^{\infty}(\Real^d)$ and 
suppose the support of $V_j$ is contained in $\overline{B}(R)$ for $j=1,\;2$. 
Suppose $V_2(x)\geq V_1(x)\geq 0$ for all $x\in \Real^d$.  Then 
$|F_{m,V_1}(i\sigma)|\leq |F_{m,V_2}(i\sigma)|$ for all $\sigma \in \Real_+$.  Moreover,
if $\sigma \geq 2(\|V_2\|_{\infty}+1)$, then 
$|F_{m,-V_1}(i\sigma)|\leq |F_{m,-V_2}(i\sigma)|$.
\end{lemma}
\begin{proof}
For any compactly supported $V\geq 0$, set
\begin{equation}\label{eq:B1}
B_{1,\pm, V}(i\sigma ) = (I\pm \vhalf R_0(i\sigma)\vhalf)^{-1/2} \vhalf
 T(i\sigma) \vhalf (I\pm \vhalf R_0(i\sigma)\vhalf)^{-1/2}
\end{equation}
and notice that if $\sigma >0$ (for the ``$+$'' sign) or $\sigma > 2(\| V\|_{\infty} +1)$ (for the ``$-$'' sign), $B_{1,\pm V}(i\sigma )$ is a self-adjoint 
trace class
operator.   Comparing (\ref{eq:Fmrewrite}), we see
that 
$$F_{m,\pm V}(i\sigma)= \det(I \pm i m B_{1,\pm, V}(i\sigma)).$$
Hence for sufficiently large $\sigma$
\begin{align}\label{eq:detprod}
|F_{m,\pm V}(i\sigma)| &  =  \left|\prod (I+ im \lambda_j(B_{1,\pm, V}(i\sigma))) \right| \nonumber 
 \\ & = \prod \left|(I+ im \lambda_j(B_{1,\pm, V}(i\sigma)))\right|
\nonumber \\ &  = \prod 
\sqrt{1+ m^2  \lambda_j^2(B_{1,\pm, V}(i\sigma))}
\end{align}
 where $\lambda_j(B_{1,\pm, V})$ are the nonzero eigenvalues of $B_{1,\pm, V}$,
repeated according to multiplicity and 
arranged in decreasing order of magnitude: $|\lambda_1(B_{1,\pm, V})|
\geq |\lambda_2(B_{1,\pm, V})| \geq ...$.

Now we turn to $V_1$ and $V_2$, and $\sigma$ as in the statement of the lemma.  
Note that
\begin{multline*}
B_{1,\pm, V_1}(i\sigma)=  
(I \pm \vhalf_1 R_0(i\sigma) \vhalf_1)^{-1/2} \frac{\vhalf_1}{\vhalf_2}
(I \pm \vhalf_2 R_0(i\sigma) \vhalf_2)^{1/2}  B_{1, \pm , V_2}(i\sigma)  \\
\times 
(I \pm \vhalf_2 R_0(i\sigma) \vhalf_2)^{1/2}\frac{\vhalf_1}{\vhalf_2}(I \pm \vhalf_1 R_0(i\sigma) \vhalf_1)^{-1/2}.
\end{multline*}
Again we use the convention that $\vhalf_1/\vhalf_2$ is $0$ outside the
support of $V_1$.  The lemma now follows from (\ref{eq:detprod}) and 
Lemmas \ref{l:normbd} and \ref{l:eigenvaluebd}.
\end{proof}

In order to obtain the lower bounds of Proposition \ref{p:lb}, we shall
need a special case of that proposition, in which the potential is 
of the form $V(x)=\epsilon\chi_B(x)$, and $\chi_B(x)$ is the 
characteristic function of a ball centered 
at the origin.  
%That is, $\chi_B(x)$  is one if $|x|<1$ and $0$ if $|x|\geq 1$.   
To study such a special case, we will introduce
spherical coordinates in $\Real^d$ (polar coordinates in the case $d=2$).  

In spherical coordinates,
$$-\Delta = -\frac{\partial^2}{\partial r^2}-\frac{d-1}{r}\frac{\partial }{\partial r} +\frac{1}{r^2}\Delta_{\Sphere^{d-1}}.$$
The eigenvalues of of the Laplacian on 
$\Sphere^{d-1}$, $\Delta_{\Sphere^{d-1}}$,
 are $l(l+d-2)$, $l\in \Natural_0$ with  multiplicity 
$$\mu(l)=\frac{2l +d-2}{d-2}\left( \begin{array}{c} l+d-3\\ 
d-3 \end{array}\right) = \frac{2l^{d-2}}{(d-2)!}(1+O(l^{-1})).$$
Denote by $Y_l^\mu$, $1\leq \mu \leq \mu(l)$, $l=0,1,2,...$ a complete
 orthonormal set of spherical 
harmonics on $\Sphere^{d-1}$ with eigenvalue $l(l+d-2)$.

We denote by $P_l$ projection onto the span of 
$$\{ h(|x|)Y^\mu_l(x/|x|): \; 1\leq \mu \leq \mu(l),\; h(|x|)\in L^2(\Real^d; r^{d-1}dr)\}.$$
Thus writing $x=r\theta$, with $r>0$ and $\theta \in \Sphere^{d-1}$
\begin{equation}
\label{eq:Pl}
(P_l g)(r\theta)=\sum_{\mu =1}^{\mu(l)} \int_{\Sphere^{d-1}} g(r\omega) 
Y^{\mu}_l(\theta)\overline{Y}^\mu_l(\omega ) d S_{\omega}.
\end{equation}

\begin{lemma}\label{l:B_1approx} Let $V\geq 0$, 
$V\in L^{\infty}_{\comp}(\Real^d)$ 
be a radial function, so that $V(x)=f(|x|)$
for some function $f\in L^{\infty}_{\comp}([0,\infty))$.  
%Let 
%$P_l$ denote projection onto the spherical harmonics with eigenvalue
%$l(l+d-2)$.  
Then for $\sigma>0$ sufficiently large, with $B_1=B_{1,\pm, V}$ 
the operator defined in (\ref{eq:B1}),
$$\left\| \large(\vhalf T(i\sigma)\vhalf -B_{1,\pm ,V}(i\sigma)\large)
P_l \right\| \leq \frac{C}{\sigma^2}\| \vhalf T(i\sigma)
\vhalf P_l\|$$
where $C$ depends on $V$ but not $\sigma$ or $l$.
\end{lemma}
\begin{proof}
To simplify the notation,
 we write $A(i\sigma)=A_{\pm, V}(i\sigma)=I\pm \vhalf R_0(i\sigma)\vhalf$, and
note that for $\sigma>0$ sufficiently large, 
\begin{equation}\label{eq:Abd}
\|A^{-1}(i\sigma)-I\|=O(1/\sigma^2),\; \|A^{-1/2}(i\sigma)-I\|=O(1/\sigma^2).
\end{equation} 
Now with $B_1$ the operator defined in (\ref{eq:B1}),
\begin{align}\label{eq:diffB1}
B_1-\vhalf  T \vhalf & = (A^{-1/2}-I)\vhalf T  
\vhalf A^{-1/2}  
+ \vhalf T \vhalf (A^{-1/2}-I).
\end{align}
Because $V$ is radial, multiplication
by either $V$ or $\vhalf$ commutes with $P_l$.  Since
$R_0$ commutes with $P_l$, so do 
$A$, $A^{-1}$, and $A^{-1/2}$. 
Thus
\begin{multline*}
\| (B_1-\vhalf  T \vhalf)P_l\|  \\ \leq 
\|(A^{-1/2}-I)\| \| \vhalf T  
\vhalf P_l\| \| A^{-1/2}\|  
+ \|\vhalf T  \vhalf P_l\| \| (A^{-1/2}-I) \|.
\end{multline*}
 Thus using (\ref{eq:Abd}) we are done.
\end{proof}

 Using the notation of \cite{ab-st}, let $J_\nu$ and $Y_\nu$ denote the Bessel
 functions of the first and second
kinds, respectively, and recall that $H^{(1)}_\nu(z)= J_\nu(z)+iY_{\nu(z)}$. 
%Let $Y^m_l $, $l=0,1,...$, $m=1,...,m(l)$ be a set of orthonormal spherical
%harmonics on $\Sphere^{d-1}$.  
For $l\in \Natural_0$, set $\nu_l= l+(d-2)/2$ and notice that 
$\nu_l$ is an integer since $d$ is even.  We can now expand $R_0(\lambda)$
using spherical harmonics.   When 
$0<\arg \lambda<\pi$ and $g\in L^2(\Real^d)$, 
\begin{equation}\label{eq:R0exp}
(R_0(\lambda)g)(r\theta)= 
\sum_{l=0}^{\infty} \sum_{\mu=1}^{\mu(l)}
\int_0^\infty  \int_{\Sphere^{d-1}} G_{\nu_l}(r,r';\lambda) Y^{\mu}_l( \theta)
\overline{Y}^\mu_l(\omega) g(r'\omega) (r')^{d-1}dS_\omega dr'
\end{equation}
with 
\begin{equation}\label{eq:Gnu}G_{\nu_l}(r,r';\lambda)=
\left\{ \begin{array}{ll}\frac{\pi}{2i}(r r')^{-(d-2)/2}
J_{\nu_l}(\lambda r)H_{\nu_l}^{(1)}(\lambda r'),\; & \text{if } 
r<r'\\\frac{\pi}{2i}(rr')^{-(d-2)/2}
H_{\nu_l}^{(1)}(\lambda r)J_{\nu_l}(\lambda r'), & \text{if}\; r\geq r'
\end{array}
\right.
\end{equation}
As noted earlier, for compactly supported, bounded $\chi$,
$\chi R_0(\lambda)\chi$ has an analytic continuation to $\Lambda$, and 
$G_{\nu_l}(r,r';\lambda)$ does as well.

Now we use \cite[9.1.35, 9.1.36]{ab-st} to obtain
$$J_\nu(e^{i\pi}z)= e^{i\pi \nu}J_{\nu}(z).$$
%Since 
%$$H_\nu^{(1)}(e^{i\pi }z) = J_\nu(e^{i\pi}z)+i Y_\nu(e^{i\pi}z)$$
Specializing \cite[9.1.36]{ab-st}  to the case of $\nu$ an integer
we have
$$Y_{\nu_l}(e^{i\pi}z)= e^{-\nu_l \pi i}(Y_{\nu_l}(z)+2i J_{\nu_l}(z))$$
giving
$$H_{\nu_l}^{(1)}(e^{i\pi }z)= e^{i\nu_l \pi}(-J_{\nu_l} (z)+i Y_{\nu_l}(z)).$$
Thus 
\begin{equation}\label{eq:gnudiff}
\tilde{G}_{\nu_l}(r,r';\lambda)\defeq
G_{\nu_l}(r,r';e^{i\pi} \lambda)-G_{\nu_l}(r,r';\lambda)
= i\pi (rr')^{-(d-2)/2}J_{\nu_l}(\lambda r)J_{\nu_l}(\lambda r').
\end{equation}
Together, (\ref{eq:R0exp}) and (\ref{eq:gnudiff}) give us an expression for
the Schwartz kernel of
 $R_0(e^{i\pi}\lambda)-R_0(\lambda)$ in spherical coordinates: with 
$r, \; r'>0$, $\theta \in \Sphere^{d-1}$, 
\begin{multline}\label{eq:shexp12}
\left((R_0(e^{i\pi \lambda})-R_0(\lambda))g\right)(r\theta)  \\=
\sum_{l=0}^{\infty} \sum_{\mu =1}^{\mu(l)}
\int_0^\infty  \int_{\Sphere^{d-1}} \tilde{G}_{\nu_l}(r,r';\lambda) Y^{\mu}_l( \theta)
\overline{Y}^\mu_l(\omega) g(r'\omega) (r')^{d-1}dS_\omega dr'.
\end{multline}

We continue to denote by $P_l$ the operator given in (\ref{eq:Pl}).
% projection onto the spherical harmonics
%with eigenvalude $l(l+d-2)$, so that
%$$(P_l g)(r\theta)=\sum_{m=1}^{m(l)} \int_{\Sphere^{d-1}} g(r\omega) 
%Y^{m}_l(\theta)\overline{Y}^m_l(\omega ) d S_{\omega}.$$  
\begin{lemma}\label{l:B1Pl} Let $B_1$ be the operator defined
in (\ref{eq:B1}).
Let $V_0=\epsilon \chi_a$, where $\epsilon,\; a>0$
and  $\chi_a$ is the characteristic function
of the ball of radius $a$ and center $0$. Fix a constant $M>3$.
 Then there is a constant $c>0$
independent of $\sigma$
so that 
$$\| B_{1,\pm ,V_0}(i\sigma)P_l\| \geq c \frac{e^{c\nu_l}}{\nu_l}$$
for all $l\in \Natural$ which satisfy $a\sigma/6>\nu_l> a \sigma /M$
for all sufficiently large $\sigma>0$.  
\end{lemma}
Before beginning the proof, we note that the constant $c$ does 
depend on $\epsilon$ and on $a$.
\begin{proof}
%Since $\vhalf_0 T\vhalf_0$ is self-adjoint and $P_l$ commutes with $T$ and
%multiplication by $\vhalf_0$, 
%$\| \vhalf_0 T V_0 T\vhalf_0 P_l\|= \| \vhalf_0 T\vhalf_0 P_l \|^2$.
From Lemma \ref{l:B_1approx} it suffices to prove an analogous lower bound
for $\| \vhalf_0 T (i\sigma)\vhalf_0 P_l\|.$

Recall
 $iT(i\sigma)= R_0(e^{i\pi }i\sigma)-R_0(i\sigma)$.
Set 
$$\psi_l(r\theta)= \chi_a(r\theta)Y^\mu_l(\theta)r^{-(d-2)/2} J_{\nu_l}(i\sigma r)$$
for any $\mu\in \{ 1,...,\mu(l)\},$
 and
note that 
$$\| \vhalf_0 T \vhalf_0 P_l\| 
\geq  \frac{\left| \langle \vhalf_0 T \vhalf_0 \psi_l, \psi_l \rangle
\right| }{\|\psi_l\|^2} .$$
By (\ref{eq:gnudiff})  and (\ref{eq:shexp12}),
\begin{align}\label{eq:lb1}
\frac{|\langle \vhalf_0 T \vhalf_0 \psi_l, \psi_l \rangle |}
{\|\psi_l\|^2} & 
= \frac{ \pi \left( \int_0^a \epsilon^{1/2} |J_{\nu_l}(i\sigma r)|^2 r^{-(d-2)}r^{d-1}dr\right)^{2}}
{ \int_0^a  |J_{\nu_l}(i\sigma r)|^2 r^{-(d-2)}r^{d-1}dr}  \nonumber \\
& = \pi \epsilon \int_0^a  |J_{\nu_l}(i\sigma r)|^2 r dr \nonumber \\
& \geq \pi \epsilon \int_{a/2}^a |J_{\nu_l}(i\sigma r)|^2 r dr.
\end{align}

%Now we choose $M>3$ as large as we wish, but fixed.
As in \cite[9.6.3]{ab-st}, setting 
$$I_\nu(z)\defeq e^{-\nu \pi i/2}J_{\nu}(z e^{i\pi/2}),\; -\pi <\arg z \leq \pi/2,$$
from \cite[9.7.7]{ab-st}  there is a constant $c>0$ 
so that for $\nu$ sufficiently large
$$|I_{\nu}(\nu s)| \geq c \frac{e^{c\nu}}{\sqrt{\nu}},\; 3\leq s \leq M.$$
Here and below we denote by $c$ a positive constant, independent of
$\nu$,  $l$, and 
$\sigma$, 
which may change from line to line.  Now we use that 
$| J_{\nu_l}(i\sigma z)|= | I_{\nu_l}(\sigma z)|$ and 
apply these to (\ref{eq:lb1}).  We find that if $3\leq \sigma r/\nu_l\leq M$
for all $r$ with $a/2\leq r \leq a$, then
$$ \left| \langle \vhalf_0 T \vhalf_0 \psi_l, \psi_l \rangle \right|
\geq 
 c \int_{a/2}^a \frac{e^{2\nu_l c}}{\nu_l} dr \geq c \frac{e^{2\nu_l c}}{\nu_l}$$
for all sufficiently large $\sigma$.  Thus, this holds for $l$ satisfying
 $a\sigma/6>\nu_l> a \sigma /M$ if $\sigma $ is sufficiently large, providing
a lower bound on $\| \vhalf_0 T \vhalf_0 P_l\|$, and thus on 
$\|B_{1,\pm V_0}(i\sigma)P_l\|$.
\end{proof}

\begin{lemma}\label{l:spcaselb}
Let $V_0=\epsilon \chi_a$, where $\epsilon, \; a>0$
and  $\chi_a$ is the characteristic function
of the ball of radius $a$ and center $0$.  Then for 
$m_0\not =0$, $m_0\in \Integers$,
 there is a $c>0$ so that for $\sigma>0$ sufficiently large
$$F_{m_0,\pm V_0}(i\sigma)\geq c\exp(c \sigma^d).$$
The constant $c$ depends on $a,\;\epsilon$, and $m_0$.
\end{lemma}
\begin{proof}
Recall that 
$$|F_{m_0,\pm V_0}(i\sigma)|= |\det( I\pm i m_0 B_{1, \pm, V_0}(i\sigma))|$$
and that for sufficiently large $\sigma>0$ $B_1(i\sigma)$ is a 
 self-adjoint operator.   Thus for sufficiently large $\sigma$
\begin{equation}\label{eq:Fmeigenvalues}
|F_{m_0,\pm V_0}(i\sigma)|= \prod_{j=1}^{\infty} \sqrt{1+m_0^2 \lambda_j^2}
\end{equation}
where $\lambda_j$ are the nonzero eigenvalues of $B_{1,\pm, V_0}(i\sigma)$. 
 The $\lambda_j$ of course depend on 
$\sigma$, but we omit this in our notation.  

A decomposition of $B_{1,\pm, V_0}$ 
using spherical harmonics shows that $B_{1,\pm, V_0}$ has 
eigenvalue $\| B_{1,\pm, V_0} P_l\|$ with multiplicity (at least) $\mu(l)$.  
Thus using (\ref{eq:Fmeigenvalues}) and the fact that $\lambda_j^2> 0$,
we get 
$$|F_{m_0}(i\sigma)|^2 
\geq \prod_{l=1}^{\infty} (1+m_0^2\| B_{1,\pm, V_0} P_l\|^2)^{\mu(l)}$$
for sufficiently large $\sigma$.  From Lemma \ref{l:B1Pl}, we see
\begin{align*}
|F_{m_0}(i\sigma)|^2 & \geq \prod_{a\sigma/6>\nu_l> a \sigma /M}
  (1+ cm_{0}^2 \frac{e^{c\nu_l}}{\nu_l^2})^{\mu(l)}\\ & 
=\exp\left( \sum_{a\sigma/6>\nu_l> a \sigma /M}\mu(l) \log
  \left(1+ c m_0^2 \frac{e^{c\nu_l}}{\nu_l^2}\right)\right) \\
& \geq \exp\left( \sum_{a\sigma/6-(d-2)/2>l> a \sigma /M -(d-2)/2} \mu(l)\left(cl
-c(d-2)/2+\log (c /\nu_l^2)
\right) \right)
\end{align*}
Now for $l$ sufficiently large, $\mu(l) \geq l^{d-2}/(d-2)!$ so we get
$$|F_m(i\sigma)|^2 \geq \exp( c \sigma^d -C)$$
for some constants $C$  and $c>0$ for all sufficiently large $\sigma$.
\end{proof}

\vspace{2mm}
\noindent {\em Proof of Proposition \ref{p:lb}.} 
We are now ready to give the proof of Proposition \ref{p:lb}.  Since
if $W$ is a translate of $V$,
$F_{m,\pm,V}= F_{m,\pm, W}$, we may assume 
$V$ can be bounded below by $V_0=\epsilon \chi_{B_a}$, where $\chi_{B_a}$ is the 
characteristic function of the ball of radius $a>0$ and center at the origin.
Then using Lemmas \ref{l:mono} and \ref{l:spcaselb} proves the proposition 
immediately.
\qed

\section{Proof of Theorem \ref{thm:lbd}}

Let $V\in L^{\infty}_{\comp}(\Real^d),$ $V\geq 0$.  We continue to 
assume $d$ is even and to
use the function 
$$F_m(\lambda)=F_{m,\pm V}(\lambda)=\det( I\pm im (1\pm \vhalf R_0(\lambda)\vhalf )^{-1} 
\vhalf T(\lambda) \vhalf)$$
 defined first by (\ref{eq:FmV}).
%For $\lambda \in \Lambda$, set 
%$$g (\lambda)=g_{m,\pm, V}(\lambda)=
% \det( I\pm im (1\pm \vhalf R_0(\lambda)\vhalf )^{-1} 
%\vhalf T(\lambda) \vhalf).$$ 
%where $\chi\in L^{\infty}_{\comp}(\Real^d)$ is a function which
%satisfies $\chi V = V$.  
Note that since 
$(I \pm \vhalf R_0(\lambda)\vhalf )^{-1}$ is a 
meromorphic function on $\Lambda$,
$F_{m,\pm V} (\lambda) $ is meromorphic on $\Lambda$.  
We shall be most interested 
in the behavior of $F_{m,\pm V}(\lambda)$ in 
$\overline{\Lambda}_0$,
since the zeros of $F_{m,\pm V}$   in $\Lambda_0$ correspond to the 
poles of $R_{\pm V}$ in $\Lambda_m$.   In the proof of Theorem \ref{thm:lbd}
we shall apply Proposition \ref{p:complex} to a function obtained by multiplying
$F_{m, \pm V}$ by a rational function.
Thus we begin this section by checking  properties of $F_{m,\pm V}$.

\begin{lemma} \label{l:gprop1}
The function $F_{m,\pm V}(\lambda)$ has only finitely many poles 
 in $\{ \lambda \in \Lambda: 0\leq \arg \Lambda \leq \pi\}$ and only finitely 
many zeros with argument $0$ or $\pi$.
\end{lemma}
\begin{proof}
We recall first  the well-known estimate 
\begin{equation}\label{eq:phyregbd}
 \| \vhalf  R_0(\lambda) \vhalf \|\leq C/|\lambda| \; \text{for}\; 
\lambda \in \Lambda, \; 0\leq \arg \lambda \leq \pi
\end{equation}  (e.g. \cite{agmon,vodeveven,vodev2}).  Thus 
for $|\lambda|\geq 2/C$, $I \pm \vhalf R_0(\lambda)\vhalf $ is 
invertible, with norm of the inverse bounded by $2$. 
Since the function $F_{m, \pm V}$ cannot have a pole at $\lambda_0$ unless
$(I\pm \vhalf R_0(\lambda)\vhalf )^{-1}$ has a pole at $\lambda_0$, we see 
$F_{m,\pm V}(\lambda)$ has 
no poles in the region $\{ \lambda \in \overline{\Lambda}_0,\;
 |\lambda|\geq r_0\}$
for some  constant $r_0$ depending on $V$.

 Moreover, from (\ref{eq:phyregbd}) $\|\vhalf T(\lambda) \vhalf  \|
\leq C/|\lambda|$ for
$\lambda \in \partial \overline{\Lambda}_0.$  Thus, there is an $r_0\geq 0$
so that $F_{m,\pm V}(\lambda)$ has no zeros
in $\{ \lambda \in \partial \overline{\Lambda}_0,\; |\lambda |\geq r_0\}$.

The bounds of Vodev \cite{vodeveven, vodev2}
 ensure that there are only finitely many poles 
of $R_{\pm V}(\lambda)$ in $\{ \lambda
\in \overline{\Lambda_{m}} : 
|\lambda|\leq r \}$ for any $r$.  Since $F_{m,\pm V}$ has a pole at $
\lambda \in 
\overline{\Lambda_0}$ only if $R_{\pm V}$ has a pole 
there, and has a zero at $z\in 
\partial\overline{\Lambda_0}$ only if $R_{\pm V}$ has a pole
at $e^{im \pi} \lambda $, this finishes the proof of the claim.
\end{proof}

\begin{lemma} \label{l:bdonreal} Let $t\in  \Lambda $ have $\arg t=0 $ or
$\arg t=\pi$.  Then
there are constants $C, \; r_0 >0$ depending on $V$ and $m$ so that 
$$\left| \frac{\frac{d}{dt} F_{m,\pm V}(t)}{F_{m, \pm V}(t)} 
\right| \leq C|t|^{d-2}\; \text{for}\; |t|\geq r_0.$$
\end{lemma}
\begin{proof}
Note that 
\begin{equation}
\label{eq:detderiv}
\frac{\frac{d}{dt}F_{m,\pm V}(t)}{F_{m, \pm V}(t)} = 
\tr \left( \pm i m (I\pm im W(t))^{-1}\frac{d}{dt}W(t) \right)
\end{equation}
where 
$$W(t)=W_{\pm V}(t)= (I\pm V^{1/2}R_0(t) V^{1/2})^{-1} V^{1/2}T(t) V^{1/2}.$$
Using (\ref{eq:phyregbd}) we see that that there is an $r_0 >0$  so that
\begin{equation}\label{eq:invbd}
\| (I \pm  \vhalf R_0(t)\vhalf)^{-1} \|\leq 2 \; \text{ for $|t|>r_0$}.
\end{equation}
For the values of $t$ in question (on the boundary of
the physical region), for any $\chi \in C_c^{\infty}(\Real^d)$
and any $j\in \Natural_0$
 there are constants $C_j$ depending on $\chi$ so that 
\begin{equation}
\left\| \frac{d^j}{dt^j} \chi R_0(t) \chi \right\| 
\leq C_j |t|^{-1-j}, \; |t|\geq 1,  
\end{equation} 
see e.g. \cite[Section 8]{j-k} or \cite[Section 16]{k-k}.
This implies that for $|t|$ sufficiently large with $\arg t=0,\; \pi$,
$\| \frac{d^j}{dt^j}W(t)\|\leq C_j$, $j=0,\; 1$, for some new 
constant $C_j$ depending on $V$.

Now we use an argument as in \cite[Lemma 3.3]{froeseodd} to bound
$\| W(t)\|_1$ and
$\| \frac{d}{dt}W(t)\|_1$, where $\| \cdot \|_1$ is the trace class norm.
We write, for $\chi \in L^{\infty}_{\comp}(\Real^d)$
\begin{equation}\label{eq:Tascomp}
\chi T(\lambda) \chi = \alpha_d \lambda^{d-2} 
{\mathbb E}_\chi^t(e^{i\pi}\lambda)  {\mathbb E}_\chi(\lambda) 
\end{equation}
where
$${\mathbb E}_\chi( \lambda):L^2(\Real^d)\rightarrow L^2(\Sphere^{d-1}),\;
 {\mathbb E}_{\chi}(\lambda)(\theta,x) =\chi(x) e^{i\lambda x \cdot \theta},\; x\in \Real^d,\; \theta 
\in \Sphere^{d-1}.$$
Then, just as in \cite{froeseodd}, we note that with $\| \cdot \|_2$ 
denoting the Hilbert-Schmidt norm,
$$\| {\mathbb E}_\chi( t ) \|^2_2=\int_{\Sphere^{d-1}} \int_{\Real^d}|e^{it\omega\cdot x}  \chi(x)|^2dxd\omega \leq C_\chi, \; \text{for}\; (\arg t)/\pi \in \Integers$$
and 
$$\left\|\frac{d}{dt} {\mathbb E}_\chi( t ) \right\|^2_2
= \int_{\Sphere^{d-1}} \int_{\Real^d}\left|i
(\omega\cdot x)e^{it\omega\cdot x}  \chi(x)\right|^2dxd\omega \leq C_\chi, \; \text{for}\; (\arg t)/\pi \in \Integers.$$
The same estimate holds for $\| {\mathbb E}_\chi^t(e^{i\pi}t)\|_2^2$
and $\| \frac{d}{dt}{\mathbb E}_\chi^t(e^{i\pi}t)\|_2^2$.
Putting this all together and using that
$\| AB\|_1\leq \|A\|_2\| B\|_2$, we see that 
$$\left \| \frac{d^j}{dt^j} W(t) \right\|_1 \leq C,\; \text{for}\; j=0,\; 1.$$
Thus 
\begin{align*}
\left| \frac{\frac{d}{dt}F_{m,\pm V}(t)}{F_{m, \pm V}(t)}\right| &  = 
\left|
\tr \left( \pm i m (I\pm im W(t))^{-1}\frac{d}{dt}W(t) \right)\right|\\
& \leq \left\|  m (I\pm im W(t))^{-1}\frac{d}{dt}W(t)\right\|_1
\leq C|t|^{d-2}
\end{align*}
when $|t|$ is sufficiently large.
%
%Using this, exactly as in the proof of 
%\cite[Lemma 3.3]{froeseodd}\footnote{Although we note that our $T(t)$ is not the same as the 
%$T$ of \cite{froeseodd}.} we can show that 
%\begin{equation}\label{eq:trclassbd}
%\left\| \vhalf \frac{d}{dt}T(t) \vhalf \right\|_{1} 
%$\leq C |t|^{d-2}, \;|t|\geq 1,\; t\in \partial \Lambda_0.
%$\end{equation}
%$Here $\| \cdot \|_1$ denotes the trace-class norm.
%Combining these two estimates (\ref{eq:intbd}) and (\ref{eq:trclassbd})
%with (\ref{eq:detderiv}) proves the lemma.
\end{proof}

The next lemma gives a bound on $F_{m,\pm V}(z)$, $z\in \Lambda_0$, which
is of a type which has been repeatedly used in proofs of upper bounds
on the number of resonances.  Closely related results can be found in
\cite{melrosepb,zworskiodd,froeseodd}, among others.
  We include the proof for the convenience of the reader,
although it is essentially a minor modification of arguments used in, for
example, \cite{zworskiodd,froeseodd} to, in the odd-dimensional case,
 bound something like the determinant of 
the scattering matrix in the physical half-plane.
\begin{lemma}
There are constants $C$, $r_0>0$  depending on $V$ and $m$ so that
$$|F_{m,\pm, V}(\lambda)|\leq C \exp(C|\lambda|^d),\; \text{for all } \lambda
 \in \overline{\Lambda_0}, \; |\lambda|>r_0.$$
\end{lemma}
\begin{proof}\label{l:upperbd}
%We begin by noting that for $\chi \in L^{\infty}_{\comp}(\Real^d)$
%$$\chi T(\lambda) \chi = \alpha_d \lambda^{d-2} 
%{\mathbb E}_\chi^t(e^{i\pi}\lambda)  {\mathbb E}_\chi(\lambda) $$
%where
%$${\mathbb E}_\chi( \lambda):L^2(\Real^d)\rightarrow L^2(\Sphere^{d-1}),\;
 %{\mathbb E}_{\chi}(\lambda)(\theta,x) =\chi(x) e^{i\lambda x \cdot \theta},\; %x\in \Real^d,\; \theta 
%\in \Sphere^{d-1}.$$
Using (\ref{eq:Tascomp}) and
that $\det(I+AB)=\det(I+BA)$ when both $AB$ and $BA$ are trace
class, 
$$F_{m,\pm,V}(\lambda)= \det( I+ K(\lambda))$$
where $K(\lambda):L^2(\Sphere^d)\rightarrow L^2(\Sphere^d)$ 
is given by 
$$K(\lambda)= \pm i m \alpha_d \lambda^{d-2} {\mathbb E}_{\vhalf}(\lambda)
 (I\pm \vhalf R_0(\lambda)\vhalf)^{-1}
{\mathbb E}_{\vhalf}^t(e^{i\pi}\lambda).  $$
Choose $r_0\geq 0$ so that 
$$\| (I\pm \vhalf R_0(\lambda)\vhalf)^{-1}\| \leq 2 \; \text{for }\; \lambda\in \Lambda_0,\; 
|\lambda|\geq r_0.$$  By slight abuse of notation, we denote the Schwartz kernel of $K$ by $K$
as well.
Then
 there is some constant $C$ so that for each $j\in \Natural$,
$$|\Delta_{\Sphere^{d-1},\theta}^j K(\lambda)(\theta,\omega) |\leq C^{2j+1}(|\lambda|^{2j}+
(2j)!)e^{C|\lambda|}\; \text{for}\; \lambda\in \overline{\Lambda_0},\; 
|\lambda|\geq r_0$$
%Since there is an $r_0\geq 0$ so that 
%$$\| (I\pm \vhalf R_0(z)\vhalf)^{-1}\| \leq 2 \; \text{for }\; z\in \Lambda_0,\; 
%|z|\geq r_0$$
since $$|\Delta_{\Sphere^{d-1}}^k e^{i \lambda x\cdot \theta} \vhalf (x) | 
\leq C^k (|\lambda|^{2k} + (2k)!)  )e^{C|\lambda|}$$ 
and 
$| (I\pm \vhalf R_0(\lambda)\vhalf)^{-1}
{\mathbb E}_{\vhalf}(e^{i\pi}\lambda)^t | \leq C \exp( C |\lambda|), $ 
when $|\lambda|\geq r_0$.
Thus by \cite[Proposition 2]{zworskiodd}, 
$$|\det (I+K(\lambda))|\leq C' e^{C'|\lambda|^{d}},\; \lambda\in 
\overline{\Lambda_0},\; |\lambda|>r_0.$$
\end{proof}

We are now ready to give the proof of Theorem \ref{thm:lbd}.
\begin{proof}
The proof is by contradiction.  So suppose for some fixed potential
$V$ satisfying the hypotheses of the theorem and for some value of 
$m\in \Integers \setminus \{ 0\}$ and for  choice of sign (positive
or negative)
\begin{equation}\label{eq:contassumpt}
\lim \sup_{r \rightarrow \infty} \frac{\log n_{m, \pm V}(r)}{\log r}<d.
\end{equation}
We work with this fixed value of $m$ and fixed choice of sign for the 
remainder of this proof.  For this choice of $m$ and sign  consider 
the function $$F_{m,\pm V}(\lambda)= \det( I\pm im (1\pm \vhalf R_0(\lambda)\vhalf )^{-1} 
\vhalf T(\lambda) \vhalf).$$
 
We denote by $\tilde{n}(r) $ the number of
zeros, counted with multiplicity, of $F_{m,\pm V}$ in $\Lambda_0$
of norm at most $r$. 
The assumption (\ref{eq:contassumpt}) means that there is a constant
$d'<d$  so that 
$n_{m,\pm V}(r) =O(r^{d'})$ for $r\rightarrow \infty$.  
Since with at most 
finitely many exceptions the zeros of $F_{m,\pm V}$ in
$\Lambda_0$ correspond, with multiplicity, to the poles of $R_{\pm V}$
in $\Lambda_m$ (see 
Section \ref{s:defF}), 
$\tilde{n}(r)=n_{m,\pm V}(r)+O(1)\leq C(1+r^{d'})$ for some 
constant $C$.

We identify $\Lambda_0$ with the upper half plane
and use the variable $z$ there.  Thus we may think of 
$F_{m,\pm V}$ as   function meromorphic in a neighborhood of 
$${\Omega}=
\{ z \in \Complex: |z|\geq 1,\; 0\leq \arg z \leq \pi\}.$$ 
Let $a_1,...,a_{m_p}$ be the poles of $F_{m,\pm V}$ in ${\Omega}$,
and let $b_1,...,b_{m_z}$ be the zeros of $ g$ in $\partial {\Omega}$, 
in both cases repeated according to multiplicity.  Recall
we know there are only finitely many by Lemma \ref{l:gprop1}.  Now
set 
$$h(z)\defeq  \frac{\prod_{j=1}^{m_p}(z-a_j) }{\prod _{j=1}^{m_z}(z-b_j)}
F_{m,\pm V}(z).$$
If there are no poles or no real zeros, the corresponding product is omitted.
By applying Lemmas \ref{l:bdonreal} and \ref{l:upperbd}, we see
that $h$ satisfies the hypotheses of 
Proposition \ref{p:complex} with $\rho=\max(d',d-1+\epsilon)$ for any
 $\epsilon>0$.   Thus for some constant $C$,
$|F_{m,\pm V}(z)| \leq C\exp(C|z|^\rho)$ for $z\in {\Omega}$ and $\rho<d$.
But this contradicts Proposition \ref{p:lb}.
\end{proof}

\end{document}